\documentclass[12pt]{amsart} 
\usepackage{amsmath,mathtools,amsthm,amsfonts,amssymb,verbatim}
\usepackage{booktabs,multicol} 
\usepackage{color, colortbl} 
\usepackage{enumitem} 
\usepackage{url}
\usepackage{tikz} 
\usepackage[letterpaper]{geometry} 
\usepackage{stmaryrd} 
\theoremstyle{definition} 
\newtheorem{theorem}[equation]{Theorem}

\newtheorem{proposition}[equation]{Proposition}

\theoremstyle{definition}
\newtheorem{definition}[equation]{Definition}

\newtheorem{rem}[equation]{Remark}
\newtheorem{exam}[equation]{Example}

\numberwithin{equation}{section}

\definecolor{mjo}{rgb}{0,0,.9}

\global\long\def\gen{\text{\sf GEN}}
\global\long\def\dng{\text{\sf DNG}}
\global\long\def\mex{\operatorname{mex}}
\global\long\def\emex{\operatorname{emex}}
\global\long\def\nim{\operatorname{nim}}
\global\long\def\opt{\operatorname{Opt}}

\global\long\def\pty{\operatorname{pty}}
\global\long\def\type{\operatorname{type}}
\global\long\def\etype{\operatorname{etype}}

\global\long\def\smo{\operatorname{smo}}

\newcommand{\notes}[1]{}

\begin{document} 
\setlength{\jot}{0pt} % this is to remove extra space between lines in aligned environment.

\title
[The spectrum of nim-values for achievement games for generating finite groups]
{The spectrum of nim-values for achievement games\\ for generating finite groups}

\author{Bret J.~Benesh}
\address{
Department of Mathematics,
College of Saint Benedict and Saint John's University,
37 College Avenue South,
Saint Joseph, MN 56374-5011, USA
}
\email{bbenesh@csbsju.edu}
\author{Dana C.~Ernst}
\author{N\'andor Sieben}
\address{
Department of Mathematics and Statistics,
Northern Arizona University PO Box 5717,
Flagstaff, AZ 86011-5717, USA
}
\email{Dana.Ernst@nau.edu, Nandor.Sieben@nau.edu}

%there are new codes for 2020:
%https://zbmath.org/static/msc2020.pdf
\subjclass[2010]{91A46, % combinatorial games
20D30% Series and lattices of subgroups
%05-08% 2020 code is 05-08Computational methods for problems pertaining to combinatorics
} 


 \date{\today}

\maketitle

%%%--------------------------------------------------------------%%%

\begin{abstract}
We study an impartial achievement game introduced by Anderson and Harary. The game is played by two players who alternately select previously unselected elements of a finite group. The game ends when the jointly selected elements generate the group. The last player able to make a move is the winner of the game. We prove that the spectrum of nim-values of these games is $\{0,1,2,3,4\}$.  This positively answers two conjectures from a previous paper by the last two authors.
\end{abstract}

%%%--------------------------------------------------------------%%%

\section{Introduction}

Anderson and Harary~\cite{anderson.harary:achievement} introduced two impartial games \emph{Generate} and \emph{Do Not Generate} in which two players alternately take turns selecting previously unselected elements of a finite group $G$. The first player who builds a generating set for the group from the jointly-selected elements wins the \emph{achievement game} $\gen(G)$.  The first player who cannot select an element without building a generating set loses the \emph{avoidance game} $\dng(G)$. The outcomes of both games were studied for some of the more familiar finite groups, including abelian, dihedral, and symmetric groups in~\cite{anderson.harary:achievement,Barnes}.

A fundamental problem in the theory of impartial combinatorial games~\cite{albert2007lessons,SiegelBook} is determining the nim-value of a game.  The nim-value determines the outcome of the game, and it also allows for the easy calculation of the nim-values of game sums.
In~\cite{ErnstSieben}, Ernst and Sieben used structure digraphs for studying the nim-values of both the achievement and avoidance games, which they applied in the context of certain finite groups including cyclic, abelian, and dihedral. Loosely speaking, a structure digraph is a quotient of the game digraph by an equivalence relation called \emph{structure equivalence}. Structure equivalence respects the nim-values of the positions of the game and drastically simplifies the calculation of the nim-values. The \emph{type} of a structure class is a triple that encodes the nim-values of the positions. Ernst and Sieben~\cite[Proposition~3.20]{ErnstSieben} determined the spectrum of types for the avoidance game $\dng(G)$, which in turn allowed them to determine that the spectrum of nim-values for $\dng(G)$ is $\{0,1,3\}$.

The goal of this paper is to determine the spectrum of nim-values for the achievement game $\gen(G)$. Our approach is very similar to that of the avoidance game, but the required calculations are significantly more difficult for groups of even order. One reason for the increased difficulty is that the game digraph of the avoidance game is a subgraph of the the game digraph of the achievement game, and hence the achievement game has more positions than the avoidance game. As a result, the structure digraphs for achievement games can be more complex. Moreover, the types associated to structure classes no longer suffice since types contain insufficient information to be closed under type calculus. To overcome this apparent shortcoming, we introduce the \emph{extended type} of a structure class, which adds a fourth component to the existing type. To analyze the behavior of the structure digraphs together with the associated extended types, we develop several type restrictions and then rely on computer calculations to handle the large number of cases. We prove that the spectrum of nim-values for the achievement game $\gen(G)$ is $\{0,1,2,3,4\}$, which positively answers Conjectures~4.8 and~4.9 from \cite{ErnstSieben}.

The structure of the paper is as follows.  We start with some preliminaries from~\cite{BeneshErnstSiebenSymAlt,BeneshErnstSiebenDNG,BeneshErnstSiebenGeneralizedDihedral,ErnstSieben}, and follow with a short characterization of the spectrum of $\gen(G)$ for $G$ of odd order.  The bulk of the work is spent on characterizing the spectrum of $\gen(G)$ for $G$ of even order.

\section{Preliminaries}

We now give a more precise description of our game. We also recall some definitions and results from \cite{BeneshErnstSiebenGeneralizedDihedral, ErnstSieben}. The positions of $\gen(G)$ are the possible sets of jointly selected elements. The starting position is the empty set. The options of a nonterminal position $P$ are of the form $P\cup\{g\}$ for some $g\in G\setminus P$. The set of options of $P$ is denoted by $\opt(P)$. 

The \emph{nim-value} of a position $P$ is recursively defined by 
\[
\nim(P)=\mex\{\nim(Q)\mid Q\in \opt(P)\},
\]
where the \emph{minimum excludant} $\mex(S)$ is the smallest nonnegative integer missing from $S$. The terminal positions of the game have no options, and so their nim-value is $\mex(\emptyset)=0$. The winning positions for the player who is about to move ($N$-positions) are those with nonzero nim-value. The winning strategy always moves the opponent into a position with zero nim-value.

\subsection{Type calculus}

The set $\mathcal{M}$ of maximal subgroups of $G$ plays an important role in this game. For a position $P$ we let
\[
\lceil P \rceil := \bigcap\{M\in\mathcal{M} \mid P\subseteq M\}.
\]
We use the simplified notation $\lceil P,g_1,\ldots, g_n \rceil$ for $\lceil P\cup \{g_1,\ldots, g_n\} \rceil$. If $P$ is a terminal position of the game, then $P$ is a generating set of $G$, and so $\lceil P \rceil=\bigcap\emptyset=G$. Note that $\lceil \emptyset \rceil=\bigcap \mathcal{M}$ is the Frattini subgroup $\Phi(G)$. 

Two positions $P$ and $Q$ are \emph{structure equivalent} if $\lceil P \rceil=\lceil Q \rceil$. Structure equivalence is an equivalence relation.
The maximum element of the equivalence class of $P$ is $\lceil P \rceil$, so we denote the \emph{structure class} of $P$ by $X_I$ where $I=\lceil P \rceil$. The set of equivalence classes is denoted by $\mathcal{D}$. 

The option relationship between positions is compatible with structure equivalence \cite[Corollary 4.3]{ErnstSieben}, so we say $X_J$ is an option of $X_I$ if $Q\in \opt(P)$ for some $P\in X_I$ and $Q\in X_J$. The set of options of $X_I$ is denoted by $\opt(X_I)$.

The vertices of the \emph{structure digraph} are the structure classes. The arrows of this digraph connect structure classes to their options.

The parity of an integer $n$ is $\pty(n):=n\text{ mod }2$. By~\cite[Proposition 4.4]{ErnstSieben}, two positions in a structure class of the same parity have the same nim-value.  We can capture this information by defining the \emph{type} of a structure class $X_I$ to be
\[
\type(X_I):=(\pty(|I|), \nim(P),\nim(Q)),
\]
where $P,Q\in X_I$ with $\pty(|P|)=0$ and $\pty(|Q|)=1$. As shown in~\cite{ErnstSieben}, each structure class contains positions of both odd and even paritities. Note that $\type(X_G)=(\pty(|G|),0,0)$ and $\type(X_I)$ is an element of $\mathbb{T}:=\{0,1\} \times \mathbb{N} \times \mathbb{N}$, where $\mathbb{N}$ is the collection of nonnegative integers.
Additionally, the second component of $\type(X_{\Phi(G)})$ is the nim-value of $\gen(G)$, since the starting position $\emptyset$ is in $X_{\Phi(G)}$. We say that the \emph{parity of the structure class} $X_I$ is the parity of $|I|$. The sets of even and odd structure classes are denoted by $\mathcal{E}$ and $\mathcal{O}$, respectively.  Thus, $\mathcal{D}=\mathcal{E} \dot\cup \mathcal{O}$.

Let $\pi_i:\mathbb{N}^n\to\mathbb{N}$ and $\tilde\pi_i:\mathbb{N}^n\to\mathbb{N}^{n-1}$ denote projection functions defined by 
\begin{align*}
\pi_i(x_1,\ldots,x_n) & :=x_i, \\
\tilde\pi_i(x_1,\ldots,x_n) & :=(x_1,\ldots,x_{i-1},x_{i+1},\ldots,x_n).
\end{align*}
We use the standard image notation $f(A):=\{f(a)\mid a\in A\}$ if $A$ is a subset of the domain of $f$.

\begin{definition}
For $T \subseteq \mathbb{T}$, define $E_T:=\pi_2(T)$, $O_T:=\pi_3(T)$, $e_T:=\mex(O_T)$, and $o_T:=\mex(E_T)$. We also define
\begin{align*}
\mex_0(T) &:=(0,e_T,\mex(E_T\cup\{e_T\})) \\
\mex_1(T) &:=(1,\mex(O_T\cup\{o_T\}),o_T).
\end{align*} 
We refer to this computation as \emph{type calculus}.
\end{definition}

The following consequence of \cite[Corollary~4.3, Proposition~4.4]{ErnstSieben} is our main tool to compute nim-values.

\begin{proposition}\label{prop:TypeCalculation} 
If $X_I\in\mathcal{D}$, then
\[
\type(X_I)=
\begin{cases}
\mex_0(\type(\opt(X_I)), & \text{$|I|$ is even} \\
\mex_1(\type(\opt(X_I)), &\text{$|I|$ is odd}. \\
\end{cases}
\]
\end{proposition}

\begin{exam}
Let $I$ have odd order and $X_I$ have options with types $(0,1,2)$ and $(1,4,3)$. Then $E_T=\{1,4\}$ and $O_T=\{2,3\}$. So 
$$
\type(X_I)=\mex_1(\{(0,1,2),(1,4,3)\})=(1,1,0)
$$ 
since the odd positions in $X_I$ have nim-value $o_T=\mex(\{1,4\})=0$, while the even positions in $X_I$ have nim-value $\mex(\{2,3,o_T\})=1$. 
\end{exam}

The \emph{deficiency} of a subset $P$ of $G$ is the minimum size $\delta_G(P)$ of a subset $Q$ of $G$ such that $\langle P\cup Q \rangle=G$. Structure equivalent positions have equal deficiencies \cite[Proposition 3.2]{BeneshErnstSiebenGeneralizedDihedral}. We define 
\begin{align*}
&\mathcal{D}_k:=\{X_I\in \mathcal{D}\mid \delta_G(I)=k\},
&&\mathcal{E}_k:=\mathcal{E}\cap \mathcal{D}_k,
&&\mathcal{O}_k:=\mathcal{O}\cap \mathcal{D}_k,
\\
&\mathcal{D}_{\geq k}:=\bigcup\{\mathcal{D}_i \mid i \geq k\},
&&\mathcal{E}_{\geq k}:= \mathcal{E} \cap\mathcal{D}_{\geq k},
&&\mathcal{O}_{\geq k}:= \mathcal{O} \cap\mathcal{D}_{\geq k}.
\end{align*}
We write $\mathcal{D}_k(G)$ when we want to emphasize the dependence on $G$. We recursively define
\begin{align*}
\mathcal{D}_{k,0} & :=\{X_I\in \mathcal{D}_k \mid \opt(X_I) \subseteq \mathcal{D}_{k-1} \} \\
\mathcal{D}_{k,l} & :=\{X_I\in \mathcal{D}_k \mid \opt(X_I) \subseteq \mathcal{D}_{k-1} \cup \mathcal{D}_{k, l-1} \}
\end{align*}
for $k,l\ge 1$. It is easy to check that the union of the nested collection $\mathcal{D}_{k,0}\subseteq \mathcal{D}_{k,1}\subseteq \mathcal{D}_{k,2}\subseteq \cdots$ is $\mathcal{D}_k$.

We visualize the structure digraph of $\gen(G)$ with a structure diagram. In a \emph{structure diagram}, vertices are denoted by triangles or circles. A structure class with even or odd parity is represented by a triangle with a flat bottom or flat top, respectively. A structure class with an unknown or unimportant parity is represented by a circle. We use several arrow types to indicate whether a change in deficiency occurs between a structure class and its option. A summary of these symbols is shown in Figure~\ref{fig:sds}. Note that Proposition~\ref{prop:Rule0EverythingHasALowerDeficiencyOption} justifies that no other arrow types are necessary.

\begin{figure}[h]
\includegraphics{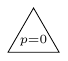}
\includegraphics{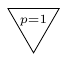}
\includegraphics{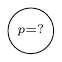}
\qquad
\includegraphics{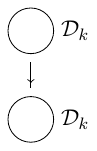}  
\includegraphics{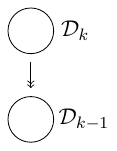}
\includegraphics{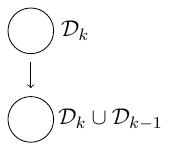}    
\caption{
\label{fig:sds}
Structure diagram symbols with $p$ denoting the parity of the structure class. The three different arrow types indicate whether the deficiency is unchanged, reduced by 1, or unspecified, respectively.}
\end{figure}

\subsection{Extended type calculus}

A further complication is that some of our restrictions require information about the even options of $X_I$, so we need to include this information in our type calculus. This motivates the following.

\begin{definition}
For $X_{I}\in\mathcal{D}_{k}$, the \emph{smoothness} of $X_I$ is
\[
\smo(X_I)=\begin{cases}
2 & \text{if } \pty(X_I)=0\\
1 & \text{if } \pty(X_I)=1 \text{ and } \opt(X_{I})\cap\mathcal{E}_{k}\neq\emptyset\\
0 & \text{otherwise}.
\end{cases}
\]
We say that $X_I$ is \emph{smooth} if $\smo(X_I)\ge 1$ and \emph{rough} otherwise.
\end{definition}

Note that an even structure class is always smooth, while the smoothness of an odd structure class depends on whether it has an even option with the same deficiency. The smoothness of an even structure class plays no role in our computations. We only define it to make the extended type in the next definition always a quadruple. This simplifies our formulas.

\begin{definition}
The \emph{extended type} of $X_I$ is $\etype(X_{I}):=(\type(X_I),\smo(X_I))$.
\end{definition}

Note that $\etype(X_I)$ is an element of $\mathbb{E}:=\mathbb{T} \times \{0,1,2\}$, although we will typically write extended types flattened as a quadruple $(p,e,o,s)$.

In an \emph{extended structure diagram}, we also indicate the smoothness of the structure classes. Smooth odd structure classes are drawn with
a double solid boundary while rough odd structure classes are drawn with a single dotted boundary. A summary of these symbols is shown in Figure~\ref{esds}.

\begin{figure}[h]
\includegraphics{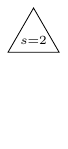} \qquad
\includegraphics{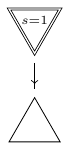}
\includegraphics{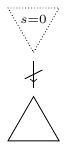}
\includegraphics{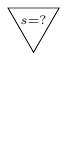}
\caption{
\label{esds}
Extended structure diagram symbols for structure classes. For odd structure classes, we use a double solid boundary if $X_I$ is smooth ($s=1$), a single dotted boundary if $X_I$ is rough ($s=0$), and single solid boundary if the smoothness is unknown or unimportant.}
\end{figure}

\begin{definition}
For $(A,B)\in\mathcal{P}(\mathbb{E})\times\mathcal{P}(\mathbb{E})$ we define 
\begin{align*}
\emex_0(A,B) &:= (\mex_0(\tilde\pi_4(A \cup B)),2), \\
\emex_1(A,B) &:= (\mex_1(\tilde\pi_4(A \cup B)),1-\min(\pi_1(A))).
\end{align*}
We refer to this computation as \emph{extended type calculus}.
\end{definition}

We think of these two functions as ways of finding the extended type of $X_I \in \mathcal{D}_n$, either real or hypothetical.  The first input $A$ consists of the extended types of the options of $X_I$  in $\mathcal{D}_n$, while the second input $B$ consists of the extended types of the options of $X_I$ in $\mathcal{D}_{n-1}$.

Extended type calculus allows us to recursively compute the extended types of every structure class, starting from the terminal structure class. 

\begin{exam}
Figure~\ref{fig:Z6} depicts the extended structure diagram for $\gen(\mathbb{Z}_6)$. The maximal subgroups are $\langle 2 \rangle$ and $\langle 3 \rangle$. The structure classes are $X_{\langle 1 \rangle}\in\mathcal{E}_0$, $X_{\langle 3 \rangle}\in\mathcal{E}_1$, and $X_{\langle 2 \rangle},X_{\langle 0 \rangle}\in\mathcal{O}_1$. 
Note that 
$\mathcal{D}_{1,0}=\{X_{\langle 3 \rangle},X_{\langle 2 \rangle}\}$ and $\mathcal{D}_{1,1}=\{X_{\langle 3 \rangle},X_{\langle 2 \rangle},X_{\langle 0 \rangle}\}$.  Extended type calculus can be used, for example, to compute
\begin{align*}
\etype(X_{\langle 0 \rangle}) 
&=\emex_1(\etype(\{X_{\langle 2 \rangle},X_{\langle 3 \rangle}\}),\etype(\{X_{\langle 1 \rangle}\})) \\
&=\emex_1(\{(1,2,1,0),(0,1,2,2)\},\{(0,0,0,2)\}) \\ &=(1,4,3,1).
\end{align*} The structure class $X_{\langle 0 \rangle}$ is smooth while $X_{\langle 2 \rangle}$ is rough.
The nim-value of the game is 
\[
\nim(\gen(\mathbb{Z}_6))=\nim(\emptyset)=\pi_2(\etype(X_{\lceil \emptyset \rceil}))=\pi_2(\etype(X_{\langle 0 \rangle}))=\pi_2(1,4,3,1)=4.
\]
\end{exam}

\begin{figure}[h]
\includegraphics[scale=1]{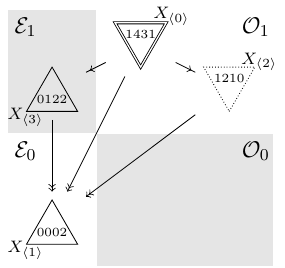}
\caption{\label{fig:Z6}Extended structure diagram for $\gen(\mathbb{Z}_6)$. The quadruples insides the triangles are the corresponding extended types.}
\end{figure}

\subsection{Some known option-type restrictions}

The following three results follow from~\cite[Proposition~3.8]{BeneshErnstSiebenGeneralizedDihedral}, Lagrange's Theorem, and~\cite[Proposition~3.9]{BeneshErnstSiebenGeneralizedDihedral}, respectively.  

\begin{proposition}\label{prop:Rule0EverythingHasALowerDeficiencyOption}
If $X_{I}\in\mathcal{D}_{k}$ for some $k\ge1$, then $\opt(X_{I})\subseteq\mathcal{D}_{k-1}\cup \mathcal{D}_{k}$
and $\opt(X_{I})\cap\mathcal{D}_{k-1}\ne\emptyset$.
\end{proposition}

The previous statement is depicted in Figure~\ref{fig:sds}. It essentially restricts the possible arrow types between structure classes.

\begin{proposition}\label{prop:Rule1EvenOptionsOnlyHaveEvenOptions}
If $X_{I}\in\mathcal{E}$ then $\opt(X_{I})\subseteq\mathcal{E}$.
\end{proposition}

This means that an even structure class has only even options, as shown in Figure~\ref{rules}(a).

\begin{proposition}\label{prop:Rule2EverythingHasAnEvenOption}
If $G$ is a group of even order and $X_I$ has an option, then $X_I$ has an even option.
\end{proposition}

The previous statement is depicted in Figure~\ref{rules}(b).

\section{Groups of odd order}

The type of a structure class can be determined relatively easily if $G$ has odd order.
The following theorem is an extension of \cite[Theorem~4.7]{ErnstSieben} and has a proof that is very similar to the proof of~\cite[Proposition~3.10]{BeneshErnstSiebenGeneralizedDihedral}. Note that we implicitly use Proposition~\ref{prop:TypeCalculation} in the following proof, as well as throughout the rest of the paper.

\begin{proposition}
\label{prop:OrderGIsOdd}
If $G$ is a group of odd order, then 
\[
	   \type(X_I) = 
	     \begin{dcases}
	       (1,0,0), & X_I \in \mathcal{O}_0\\
	       (1,2,1), & X_I \in \mathcal{O}_1\\
	       (1,2,0), & X_I \in \mathcal{O}_2\\
	       (1,1,0), & X_I \in \mathcal{O}_{\geq 3}.
	     \end{dcases}
\]
\end{proposition}

\begin{proof}
We will use structural induction on the structure classes.  By Proposition~\ref{prop:Rule0EverythingHasALowerDeficiencyOption} and Lagrange's Theorem, $X_I\in\mathcal{O}_m$ for $m\ge 1$ implies $\opt(X_I)\subseteq\mathcal{O}_{m}\cup\mathcal{O}_{m-1}$ and $\mathcal{O}_{m-1}\cap\opt(X_I)\ne\emptyset$.

If $X_I \in \mathcal{O}_0$, then $\type(X_I)=(1,0,0)$ since  $I=G$.
If $X_I \in \mathcal{O}_1$, then $\type(X_I)=(1,2,1)$ since
\[
\type(\opt(X_I))
=
\begin{dcases}
\{(1,0,0)\} & \text{if } \opt(X_I)\subseteq\mathcal{O}_0 \\
\{(1,0,0),(1,2,1)\} & \text{otherwise}
\end{dcases}
\]
by induction. If $X_I \in \mathcal{O}_2$, then $\type(X_I)=(1,2,0)$ since
\[
\type(\opt(X_I))
=
\begin{dcases}
\{(1,2,1)\} & \text{if }  \opt(X_I)\subseteq\mathcal{O}_1 \\
\{(1,2,1),(1,2,0)\} & \text{otherwise}
\end{dcases}
\]
by induction. If $X_I  \in \mathcal{O}_{3}$, then $\type(X_I)=(1,1,0)$ since
\[
\type(\opt(X_I))
=
\begin{dcases}
\{(1,2,0)\} & \text{if } \opt(X_I)\subseteq\mathcal{O}_2 \\
\{(1,2,0),(1,1,0)\} & \text{otherwise}
\end{dcases}
\]
by induction. If $X_I \in \mathcal{O}_{\geq 4}$, then $\type(X_I)=(1,1,0)$, since every option of $X_I$ has type $(1,1,0)$ by induction.
\end{proof}

\section{Groups of even order}

Our main goal in this section is to compute the possible nim-values of $\gen(G)$ for a group $G$ of even order. Our approach is similar to that of Proposition~\ref{prop:OrderGIsOdd}. We want to recursively build all possible types of structure classes with a given deficiency from the already-computed types with lower deficiency.
Unfortunately this simple approach is not sufficient to complete this computation, because it quickly becomes unwieldy for groups of even order as it yields an infinite number of potential types.  However, we can use group theory to impose restrictions on the type calculations, which will reduce the number of potential types by eliminating many types that are not possible. We already have three of these restrictions: Propositions~\ref{prop:Rule0EverythingHasALowerDeficiencyOption}, \ref{prop:Rule1EvenOptionsOnlyHaveEvenOptions}, and \ref{prop:Rule2EverythingHasAnEvenOption}. In this section, we develop additional restrictions involving smoothness, which is the reason why we introduced extended types. We then use these restrictions to carry out the computation on extended types using the algorithm in Subsection~\ref{subsec:algorithm}.

\subsection{Additional option-type restrictions}

In this subsection we present two option type restrictions that involve smoothness. A diagrammatic depiction of the statements are shown in Figures~\ref{rules}(c) and \ref{rules}(d), respectively.

\begin{proposition}\label{prop:ForcedEvenOptions1}
Let $X_I, X_J \in \mathcal{O}_n$ such that $X_J$ is an option of $X_I$.  
If $X_J$ is smooth, then so is $X_I$.
\end{proposition}

\begin{proof}
Suppose that $X_J$ has an option in $\mathcal{E}_n$, as shown in Figure~\ref{proofFigs}(a).  Then there is a $g \in G$ such that $X_{\lceil J,g \rceil} \in \mathcal{E}_n$.  By Cauchy's Theorem, there is an element $t$ in $\lceil J,g \rceil$ of order $2$. Since $X_I\in\mathcal{O}_n$, $t\notin I$. Then $\lceil I,t\rceil$ has even order, so $X_I$ has an option $X_{\lceil I,t\rceil}$ in $\mathcal{E}$. Since $I \leq \lceil I,t \rceil \leq \lceil J,t \rceil \leq \lceil J,g \rceil$ with both $X_I$ and $X_{\lceil J,g \rceil}$ in $\mathcal{D}_n$, we conclude that $X_{\lceil I,t \rceil} \in \mathcal{E}_n$.   Thus, $X_I$ has an option $X_{\lceil I,t \rceil}$ in $\mathcal{E}_n$.
\end{proof}

\begin{figure}[h]
\begin{tabular}{c c}
\includegraphics{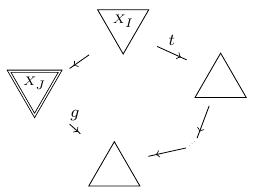}
&
\includegraphics{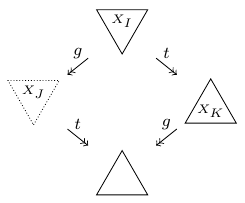}
\\
(a) & (b) 
\end{tabular}
\caption{
\label{proofFigs}
Figures for Propositions~\ref{prop:ForcedEvenOptions1} and \ref{prop:GeneralizedEvenNonContainment}.}
\end{figure}

\begin{proposition} 
\label{prop:GeneralizedEvenNonContainment}
Let $G$ be a group of even order and assume that 
$X_I \in \mathcal{O}_n$
and $X_J \in \mathcal{O}_{n-1}$ such that $X_J$ is an option of $X_I$.  
If $X_J$ is rough, then so is $X_I$.
\end{proposition}

\begin{proof}
Assume $X_L$ is an even option of $X_I$.
We will show that $X_L$ is in $\mathcal{E}_{n-1}$.
Since $L$ has even order, it contains an element $t$ of even order. Let $K:=\lceil I,t\rceil$,
as shown in Figure~\ref{proofFigs}(b).
Note that $X_K \in \mathcal{E}_n \cup \mathcal{E}_{n-1}$ by Proposition~\ref{prop:Rule0EverythingHasALowerDeficiencyOption} since $t$ has even order and $X_K$ is an option of $X_I$.
By Lagrange's Theorem, $X_{\lceil J,t\rceil} \in \mathcal{E}$. Since $X_J$ is rough, we have $X_{\lceil J,t \rceil} \not\in \mathcal{E}_{n-1}$.  Hence $X_{\lceil J,t \rceil} \in \mathcal{E}_{n-2}$ by Proposition~\ref{prop:Rule0EverythingHasALowerDeficiencyOption}. We have $J=\lceil I,g \rceil$ for some $g \in G$.    Since $X_{\lceil K,g \rceil} = X_{\lceil I,g,t \rceil} = X_{\lceil J,t \rceil} \in \mathcal{E}_{n-2}$, we conclude that $X_K \in \mathcal{E}_{n-1}$ by Proposition~\ref{prop:Rule0EverythingHasALowerDeficiencyOption}. 
Since $K$ is a subgroup of $L$, $X_L\in\mathcal{E}_{n-1}$, as well.
\end{proof}

\begin{figure}[h]
\begin{tabular}{c c c c}
\includegraphics{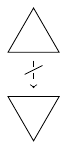} &
\includegraphics{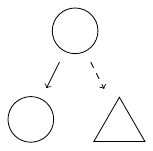} &
\includegraphics{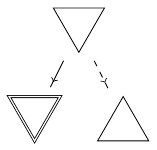} &
\includegraphics{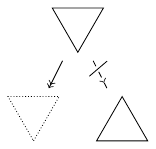} \\
(a) Proposition~\ref{prop:Rule1EvenOptionsOnlyHaveEvenOptions} & (b) Proposition~\ref{prop:Rule2EverythingHasAnEvenOption} & (c) Proposition~\ref{prop:ForcedEvenOptions1} & (d) Proposition~\ref{prop:GeneralizedEvenNonContainment}
\end{tabular}
\caption{
\label{rules}
Diagrams for option type restrictions. As with commutative diagrams, solid arrows are assumed to exist, indicating premises, while the dashed arrows are guaranteed to exist, indicating conclusions. A crossed-out dashed arrow is guaranteed not to exist.
}
\end{figure}

\subsection{Spectrum of extended types}
\label{subsec:algorithm}

The next definition introduces the spectrum of extended types for groups of even order.

\begin{definition}
For $k,l\ge 0$ we let $E_k:=\bigcup_l E_{k,l}$, where
\begin{align*}
E_{k,l} & :=\bigcup\{ \etype(\mathcal{D}_{k,l}(G)) \mid G \text{ is a group of even order}\},
\end{align*} 
so that $E_k =\bigcup\{ \etype(\mathcal{D}_k(G)) \mid G \text{ is a group of even order}\}$. We also define $E:=\bigcup_k E_k$ to be the spectrum of extended types of groups of even order. 
\end{definition}

Determining $E$ appears to be difficult, so we define the set of feasible extended types. The set of feasible extended types is easier to compute and turns out to be a superset of the spectrum of extended types. 
Recall that the nim-value of a game occurs as the second component of some extended type.
This larger set of feasible extended types does not introduce extraneous nim-values because Example~\ref{ex:nimExamplesParity} demonstrates that we can find examples of groups with each of the nim-values.

The next definition reformulates Propositions~\ref{prop:Rule0EverythingHasALowerDeficiencyOption} and \ref{prop:Rule1EvenOptionsOnlyHaveEvenOptions}
in the language of extended types. 

\begin{definition}
A pair $(A,B)$ in $\mathcal{P}(\mathbb{E})\times\mathcal{P}(\mathbb{E})$ is \emph{$0$-feasible} if $B\neq \emptyset$ and $1\not\in\pi_1(A\cup B)$.
\end{definition}

The four criteria in the following definition are reformulations of Propositions~\ref{prop:Rule0EverythingHasALowerDeficiencyOption},
\ref{prop:Rule2EverythingHasAnEvenOption}, \ref{prop:ForcedEvenOptions1}, and~\ref{prop:GeneralizedEvenNonContainment}, respectively. 

\begin{definition}
A pair $(A,B)$ in $\mathcal{P}(\mathbb{E})\times\mathcal{P}(\mathbb{E})$ is \emph{$1$-feasible} if it satisfies the following conditions:
\begin{enumerate}
\item $B\not =\emptyset$.
\item $0\in\pi_1( A\cup B )$.
\item $1\in\pi_4(A)$ implies $0\in\pi_1(A)$.
\item $0\in\pi_4(B)$ implies $0\not\in\pi_1(A)$.
\end{enumerate}
\end{definition}

We are ready to define our approximation to the $E_k$. 

\begin{definition}
\label{Ebardef}
We let $\bar{E}_{0} := \{(0,0,0,2)\}$. For $k,l\ge 1$, we recursively define
\begin{align*}
\bar{E}_{k,0} & :=\{\emex_p(\emptyset,B) \mid p\in\{0,1\}, B\in\mathcal{P}(\bar E_{k-1}), (\emptyset,B) \text{ is $p$-feasible} \},  \\
\bar{E}_{k,l} & :=\{\emex_p(A, B) \mid p\in\{0,1\}, (A,B)\in\mathcal{P}(\bar E_{k,l-1})\times\mathcal{P}(\bar E_{k-1}), (A,B) \text{ is $p$-feasible} \}, \\
\bar{E}_{k}   & :=\bigcup\{ \bar{E}_{k,l} \mid l\in\mathbb{N} \}.
\end{align*}
We also define $\bar E:=\bigcup_k \bar{E}_k$ to be the \emph{feasible spectrum} of extended types of groups of even order.
\end{definition}

The reason why we are distinguishing between $E_k$ and $\bar{E}_k$ is that $\bar{E}_k$ may contain extended types that cannot exist for an actual group.

\begin{proposition}
For $k\geq 1$, $\bar{E}_{k,0}\subseteq \bar{E}_{k,1}\subseteq \bar{E}_{k,2}\subseteq\cdots \subseteq \bar{E}_{k}$.
\end{proposition}

\begin{proof}
We will prove that $\bar{E}_{k,l} \subseteq \bar{E}_{k,l+1}$ by induction on $l$, and it is clear that $\bar{E}_{k,l} \subseteq \bar{E}_k$ by definition of $\bar{E}_k$.

Let $t \in \bar{E}_{k,0}$.  Then $t=\emex_p(\emptyset,B)$ for some $p \in \{0,1\}$ and $B \in \mathcal{P}(\bar{E}_{k-1})$ such that $(\emptyset,B)$ is $p$-feasible.  Since $\emptyset \in \mathcal{P}(\bar{E}_{k,0})$, we have $t=\emex_p(\emptyset, B) \in \bar{E}_{k,1}$. 

Now suppose that $\bar{E}_{k,l-1} \subseteq \bar{E}_{k,l}$, and let $r \in \bar{E}_{k,l}$.  Then $r=\emex_p(A, B)$ for some $p \in \{0,1\}$, $A \in \mathcal{P}(\bar{E}_{k,l-1})$, and $B \in \mathcal{P}(\bar{E}_{k-1})$ such that $(A,B)$ is $p$-feasible.  Since $\bar{E}_{k,l-1} \subseteq \bar{E}_{k,l}$ by induction, we have $\mathcal{P}(\bar{E}_{k,l-1}) \subseteq \mathcal{P}(\bar{E}_{k,l})$ and so $A \in \mathcal{P}(\bar{E}_{k,l})$.  Then $r=\emex_p(A, B) \in \bar{E}_{k,l+1}$, and we conclude that $\bar{E}_{k,l} \subseteq \bar{E}_{k,l+1}$.
\end{proof}

The next result shows that every extended type that actually occurs in a group is a feasible extended type.  This is no surprise.  Both $E_k$ and $\bar E_k$ are recursively computed in the same way with the $\emex$ function, although the construction of the extended types that actually occur may have additional restrictions on the input than the construction of the feasible extended types.  Thus, the creation of $E_k$ is a possibly more restrictive process, so it must be a subset of $\bar E_k$.

\begin{proposition}
For all $k\in \mathbb{N}$, $E_{k} \subseteq \bar{E}_{k}$.
\end{proposition} 
\begin{proof}
For a contradiction, assume there is a least $k$ such that $E_k \not\subseteq \bar E_k$.  Since $E_0  = \{(0,0,0,2)\}=\bar E_0$, we may assume that $k \geq 1$.  Then there must be a least $l$ such that $E_{k,l} \not\subseteq \bar E_{k,l}$, and let $t \in E_{k,l} \setminus \bar E_{k,l}$.  Since $t \in E_k$, there is a finite group $G$ of even order and structure class $X_I \in \mathcal{D}_{k,l}(G)$ such that $t=\etype(X_I)$.  Then $t = \emex_p(A,B)$ where $A:=\etype(\opt(X_I) \cap  \mathcal{D}_{k,l-1}(G))$, $B:=\etype(\opt(X_I) \cap \mathcal{D}_{k-1})$, and $p:=\pty(|I|)$. By Propositions~\ref{prop:Rule0EverythingHasALowerDeficiencyOption}, \ref{prop:Rule1EvenOptionsOnlyHaveEvenOptions}, \ref{prop:Rule2EverythingHasAnEvenOption}, \ref{prop:ForcedEvenOptions1}, and~\ref{prop:GeneralizedEvenNonContainment}, we have that $(A,B)$ is $p$-feasible.  Additionally, \[B \subseteq \etype(\mathcal{D}_{k-1}(G)) \subseteq E_{k-1} \subseteq \bar E_{k-1}, \]  by the choice of $k$.

If $l=0$ then $t=\emex_p(A,B)=\emex_p(\emptyset,B) \in \bar E_{k,0}$, a contradiction. Thus, we may assume that $l \geq 1$. Then \[A =\etype(\opt(X_I) \cap \mathcal{D}_{k,l-1}(G)) \subseteq \etype(\mathcal{D}_{k,l-1}(G)) \subseteq E_{k,l-1} \subseteq \bar E_{k,l-1},\] 
by the choice of $l$. Thus, $t=\etype(X_I)=\emex_p(A,B) \in \bar E_{k,l}$, a contradiction.
\end{proof}

\begin{table}
\centering
\setlength{\fboxsep}{1.5pt}
\setlength{\tabcolsep}{0.35em}
\begin{tabular}{@{}cc@{}}
\toprule
\addlinespace[0.25em] $\etype\in\bar E_k$  & $k$\\
\midrule
\vphantom{$\framebox(9,9.5){2}$}$(0,0,0,2)$ & $0$\\
\vphantom{$\framebox(9,9.5){2}$}$(0,1,2,2)$ & $1$\\
\vphantom{$\framebox(9,9.5){2}$}$(1,1,2,1)$ & $1,2$\\
\vphantom{$\framebox(9,9.5){2}$}$(1,2,1,0)$ & $1$\\
\vphantom{$\framebox(9,9.5){2}$}$(1,4,3,1)$ & $1$\\
\bottomrule
\end{tabular}
\hfil
\begin{tabular}{@{}cc@{}}
\toprule
\addlinespace[0.25em] $\etype\in\bar E_k$  & $k$ \\
\midrule
\vphantom{$\framebox(9,9.5){2}$}$(0,0,2,2)$ & $2$\\
\vphantom{$\framebox(9,9.5){2}$}$(1,0,1,1)$ & $\fbox{2},3,\ldots$\\
\vphantom{$\framebox(9,9.5){2}$}$(1,0,2,1)$ & $2,3$\\
\vphantom{$\framebox(9,9.5){2}$}$(1,1,0,0)$ & $2$\\
\vphantom{$\framebox(9,9.5){2}$}$(1,3,0,0)$ & $2$\\
\bottomrule
\end{tabular}
\hfil
\begin{tabular}{@{}cc@{}}
\toprule
\addlinespace[0.25em] $\etype\in\bar E_k$  & $k$ \\
\midrule
\vphantom{$\framebox(9,9.5){2}$}$(1,3,2,1)$ & $\fbox{2}$\\
\vphantom{$\framebox(9,9.5){2}$}$(1,4,0,0)$ & $2$\\
\vphantom{$\framebox(9,9.5){2}$}$(1,4,1,1)$ & $\fbox{2}$\\
\vphantom{$\framebox(9,9.5){2}$}$(1,4,2,1)$ & $\fbox{2}$\\
\vphantom{$\framebox(9,9.5){2}$}$(0,0,1,2)$ & $3,4,\ldots$\\
\bottomrule
\end{tabular}
\hfil 
\begin{tabular}{@{}cc@{}}
\toprule
\addlinespace[0.25em] $\etype\in\bar E_k$  & $k$ \\
\midrule
\vphantom{$\framebox(9,9.5){2}$}$(1,0,1,0)$ & $3,4,\ldots$\\
\vphantom{$\framebox(9,9.5){2}$}$(1,0,2,0)$ & $\fbox{3},4$\\
\vphantom{$\framebox(9,9.5){2}$}$(1,1,2,0)$ & $3$\\
\vphantom{$\framebox(9,9.5){2}$}$(1,3,1,0)$ & $3$\\
\vphantom{$\framebox(9,9.5){2}$}$(1,3,2,0)$ & $3$\\
\bottomrule
\end{tabular}
\vspace{0.1in}
\caption{\label{table:eTypeTable} The elements of $\bar E_k$ for each deficiency $k$. We found instances of every extended type with each deficiency using a computer search except for the five with a $\fbox{\text{box}}$ around them.
}
\end{table}

\begin{proposition}\label{prop:BigProp}
The elements of $\bar E$ are the extended types shown in Table~\ref{table:eTypeTable}.
\end{proposition}

\begin{proof}
Using Definition~\ref{Ebardef}, we computed $\bar E$ using a GAP~\cite{GAP} program. The code and its output are available on the companion web page \cite{WEB3}. The results show that the computation of each $\bar E_k$ finishes in finitely many iterations. This is indicated by the equality of $\bar E_{k,l}$ and $\bar E_{k,l+1}$ for some $l$. The results also show that $\bar E_5=\bar E_6$. Hence $\bar E=\bigcup_{k=0}^5\bar E_k$ and the whole computation finishes in finitely many steps.
\end{proof}

Even though we computed $\bar E$ with a computer, we also verified the output by hand. Note that a human can eliminate many of the large number of cases that the computer checked. 

\begin{exam}
We demonstrate the computation of $\bar E_1=\bar E_{1,1}$. We have
\begin{align*}
\bar E_0&=\{(0,0,0,2)\}, \\
\bar E_{1,0}&=\{(0,1,2,2),(1,2,1,0)\}, \\
\bar E_{1,1}&=\{(0,1,2,2),(1,1,2,1),(1,2,1,0),(1,4,3,1)\}, \\
\bar E_{1,2}&=\bar E_{1,1}.
\end{align*}
For example $(0,1,2,2)=\emex_0(\emptyset,\bar E_0)$ and $(1,4,3,1)=\emex_1(\bar E_{1,0},\bar E_0)$. Note that $(\emptyset,\bar E_0)$ is 0-feasible and $(\bar E_{1,0},\bar E_0)$ is 1-feasible. Also, note that $\bar E_1=\bar E_{1,1}$ since $\bar E_{1,2}=\bar E_{1,1}$. 
\end{exam}

\begin{rem}
We found examples of the extended types with every deficiency shown in Table~\ref{table:eTypeTable} using a computer search except for the five listed with a box around them. For instance, it is possible that $(1,0,2,0) \in \bar E_3 \setminus E_3$, but we have not found such an example. However, we have verified that $(1,0,2,0) \in E_4$ by looking at subgroups of \texttt{SmallGroup(500,48)} in GAP's \cite{GAP} SmallGroup database.
\end{rem}

\section{Spectrum of nim-values}

We are now ready to determine the spectrum of nim-values of $\gen(G)$. If the order of $G$ is odd, then Proposition~\ref{prop:OrderGIsOdd} implies that the spectrum of nim-values of $\gen(G)$ is a subset of $\{0,1,2\}$.  If the order of $G$ is even, then Proposition~\ref{prop:BigProp} and the containment 
$$
\{\pi_2(\type(X_{\Phi(G)}))\mid G\text{ is an even group}\}\subseteq \pi_2(E)\subseteq\pi_2(\bar E).
$$
shows that the spectrum of nim-values of $\gen(G)$ is a subset of $\{0,1,2,3,4\}$.  The next example verifies that we have equality in both cases.

\begin{exam}\label{ex:nimExamplesParity}
The nim-values for the following odd and even-ordered groups were computed in~\cite{ErnstSieben}. The groups listed in the table have the smallest possible order for the given parity and nim-value.

\aboverulesep=0ex
\belowrulesep=0ex
\renewcommand{\arraystretch}{1.1}

\begin{center}
\begin{tabular}{@{}c||ccc|ccccc@{}}
\toprule
$G$  & $\mathbb{Z}_{1}$ & $\mathbb{Z}_{3}^{3}$ & $\mathbb{Z}_{3}$  & $\mathbb{Z}_{2}^{3}$ & $\mathbb{Z}_{4}$ & $\mathbb{Z}_{2}$ & $S_{3}$ & $\mathbb{Z}_{6}$\\
\midrule
$\nim(\gen(G))$ & $0$ & $1$ & $2$  & $0$ & $1$ & $2$ & $3$ & $4$\\
\bottomrule
\end{tabular}
\end{center}
\end{exam}

The discussion above together with Example~\ref{ex:nimExamplesParity} immediately implies the following result.

\begin{proposition}
The spectrum of nim-values of $\gen(G)$ for groups with odd order is $\{0,1,2\}$. The spectrum of nim-values of $\gen(G)$ for groups with even order is $\{0,1,2,3,4\}$.
\end{proposition}

Now we have our main result.

\begin{theorem}
The spectrum of nim-values of the achievement game $\gen(G)$ for a finite group $G$ is $\{0,1,2,3,4\}$.
\end{theorem}

\section{Open Problems and Conjectures}

We close with a handful of open problems and conjectures.  

\begin{enumerate}
\item One can find examples of all of the extended types for each deficiency in Table~\ref{table:eTypeTable} except for the boxed $(1,0,1,1)$, $(1,3,2,1)$, $(1,4,1,1)$, $(1,4,2,1)$ types from $\bar E_2$ and $(1,0,2,0)$ from $\bar E_3$.  Do these five extended types actually occur with the appropriate deficiencies?

\item Computer experimentation shows that adding a type restriction corresponding to the following conjecture eliminates all but $(1,0,2,0)$ of the five boxed extended types from Table~\ref{table:eTypeTable}:  

\vspace{0.1in}

\begin{quote}
\emph{If $G$ is even and $k\ge 0$, then every $X_I\in\mathcal{O}_{k+1}$ has an option $X_L\in\mathcal{E}_k$.}
\end{quote}

\vspace{0.1in}

\noindent
In fact, proving this conjecture for the special case when $X_I \in \mathcal{O}_2$ would be sufficient since the remaining four boxed extended types are all in $\mathcal{O}_2$.  However, we were not able to prove this conjecture.    A natural idea for a proof of this conjecture would be to prove the stronger statement:

\vspace{0.1in}

\begin{quote}
\emph{If $G$ is even and $k\ge 0$, then for every $X_I \in \mathcal{O}_{k+1}$ there is a $t$ of even order such that $X_{\lceil I,t \rceil} \in \mathcal{E}_k$.}
\end{quote}

\vspace{0.1in}

\noindent
Unfortunately, this statement is not true. In private correspondence, Marsden Conder provided a counterexample: \texttt{SmallGroup(240,191)} in GAP's \cite{GAP} SmallGroup database, which is isomorphic to $\mathbb{Z}_2^4 \rtimes \mathbb{Z}_{15}$.  However, this is not a counterexample for the original conjecture.

\item Does a type give algebraic information about the corresponding subgroup?  For instance, does the type characterize what kind of maximal subgroups contain the subgroup? 
\item In~\cite{BeneshErnstSiebenDNG}, the authors provide a checklist in terms of maximal subgroups for determining the nim-value of $\dng(G)$. Is there an analogous set of criteria for determining the nim-value of $\gen(G)$?
\end{enumerate}

\section*{Acknowledgements}

We thank Bob Guralnick and Marston Conder for giving us thoughtful examples to consider.

\bibliographystyle{amsplain}
\bibliography{game}
\end{document}